\documentclass[12pt]{amsart}%
\usepackage{amsfonts}
\usepackage{amsmath}
\usepackage{amssymb}
\usepackage{graphicx}%
\makeatletter
\@addtoreset{equation}{section}
\makeatother

\newtheorem{theorem}{Theorem}[section]

\newtheorem{lemma}[theorem]{Lemma}

\newtheorem{remark}[theorem]{Remark}

\makeatletter
\def\@makefnmark{}
\makeatother

\begin{document}

\title[Extremal functions for Stein-Weiss inequalities on the heisenberg group]{Existence of extremal functions for the Stein-Weiss inequalities on the heisenberg group}

\author{Lu Chen}

\address{School of Mathematical and statistics, Beijing Institute of Technology, Beijing 100081, P. R. China}
\email{chenlu5818804@163.com}
\author{Guozhen Lu }
\address{Department of Mathematics\\
University of Connecticut\\
Storrs, CT 06269, USA}
\email{guozhen.lu@uconn.edu}

\author{Chunxia Tao}
\address{School  of Mathematical Sciences\\
 Beijing Normal  University\\
  Beijing 100875, China}
\email{taochunxia@mail.bnu.edu.cn}

\thanks{The first and third authors were partly supported by a grant from the NNSF of China (No.11371056), the second author was partly supported by a grant from the Simons foundation. Corresponding Authors: Guozhen Lu and Chunxia Tao.}

\begin{abstract}
In this paper, we establish the existence of extremals for two kinds of Stein-Weiss inequalities on the Heisenberg group. More precisely, we prove the existence of extremals for the  Stein-Weiss inequalities with full weights in Theorem \ref{thm1} and the  Stein-Weiss inequalities with horizontal weights in Theorem \ref{thm3}. Different from the proof of the analogous inequality in Euclidean spaces given by Lieb \cite{Lieb} using Riesz rearrangement inequality which is not available on the Heisenberg group, we employ the concentration compactness principle to obtain the existence of the maximizers on the Heisenberg group. Our result is also new even in the Euclidean case because we don't assume that the exponents of the double weights in the Stein-Weiss inequality (\ref{SW}) are both nonnegative  (see Theorem 1.3 and more generally Theorem 1.5).
Therefore, we extend Lieb's celebrated result of the existence of extremal functions of the Stein-Weiss inequality in the Euclidean space to the case where the exponents are not necessarily both nonnegative (see Theorem 1.3).  Furthermore, since the absence of translation invariance of the Stein-Weiss inequalities, additional difficulty presents and one cannot simply follow the same line of Lions' idea to obtain our desired result. Our methods can also be used to obtain the existence of optimizers for several other weighted integral inequalities (Theorem 1.5).
\end{abstract}

\maketitle {\small {\bf Keywords:}  Concentration compactness principle; Existence of extremal functions; Heisenberg group, Stein-Weiss inequalities.\\

\section{Introduction}
We recall the classical Stein-Weiss inequality on $\mathbb{R}^n$:
\begin{equation}\label{SW}
\int_{\mathbb{R}^n}\int_{\mathbb{R}^n}|x|^{-\alpha}|x-y|^{-\lambda} f(x)g(y)|y|^{-\beta} dxdy\leq C_{n,\alpha,\beta,p,q'}\|f\|_{L^{q'}(\mathbb{R}^n)}\|g\|_{L^p(\mathbb{R}^n)},
\end{equation}
where $p$, $q'$, $\alpha$, $\beta$ and $\lambda$ satisfy the following conditions,
$$\frac{1}{q'}+\frac{1}{p}+\frac{\alpha+\beta+\lambda}{n}=2,\ \ \frac{1}{q'}+\frac{1}{p}\geq 1,$$
$$1<p, q<\infty,\ \ \ \alpha+\beta\geq0,\ \ \alpha<\frac{n}{q},\ \ \beta<\frac{n}{p'},\ \ 0<\lambda<n.$$
Lieb \cite{Lieb} applied the rearrangement inequalities to establish the existence of extremals for inequality \eqref{SW} in the case $p<q$ and $\alpha, \beta \geq 0$. Furthermore, in the case of $p=q$,   the extremals can't be expected to exist (see Lieb \cite{Lieb} and also  Herbst \cite{Herbst} for the case $\lambda=n-1$, $p=q=2$, $\alpha=0$, $\beta=1$). In the case of $p=q$, Beckner \cite{B1,B2, B3} obtained the sharp constant of the Stein-Weiss inequalities \eqref{SW}. The precise estimate of the sharp constant of the Stein-Weiss inequalities for the case of $p\neq q$ was also established in \cite{B3}.

\vskip0.3cm

When $\alpha=\beta=0$, the Stein-Weiss inequality \eqref{SW} reduces to the following Hardy-Littlewood-Sobolev inequality (\cite{Hardy,Sobolev}),
\begin{equation}\label{HL1}
\int_{\mathbb{R}^n}\int_{\mathbb{R}^n}|x-y|^{-\lambda} f(x)g(y)dxdy\leq C_{n,p,q'}\|f\|_{L^{q'}(\mathbb{R}^n)}\|g\|_{L^p(\mathbb{R}^n)},
\end{equation}
where $1<q', p<\infty, 0<\lambda<n$ and $\frac{1}{q'}+\frac{1}{p}+\frac{\lambda}{n}=2$.
Lieb and Loss \cite{LiebLoss} used the layer cake representation to prove that the sharp constants $C_{n,p,q'}$ satisfy the following estimate
$$C_{n,p,q'}\leq \frac{n}{n-\lambda}\Big(\frac{\pi^{\frac{\lambda}{2}}}{\Gamma(1+\frac{n}{2})}\Big)^{\frac{\lambda}{n}}\frac{1}{q'p}
\Big(\big(\frac{\lambda q'}{n(q'-1)}\big)^{\frac{\lambda}{n}}+\big(\frac{\lambda p}{n(p-1)}\big)^{\frac{\lambda}{n}}\Big).$$
Furthermore, when $p=q'=\frac{2n}{2n-\lambda}$, Lieb \cite{Lieb} also gave the explicit formula of sharp constants and maximizers. More precisely,
\medskip

\noindent\textbf{Theorem A.}
For $1<q'$, $p<\infty$, $0<\lambda<n$ and $\frac{1}{q'}+\frac{1}{p}+\frac{\lambda}{n}=2$, there exists sharp constants $C_{n,p,q'}$ and maximizers of $f\in L^{q'}(\mathbb{R}^n)$ and $g\in L^p(\mathbb{R}^n)$, such that

\begin{equation}\label{EHL}
\int_{\mathbb{R}^n}\int_{\mathbb{R}^n} f(x)|x-y|^{-\lambda}g(y)dxdy\leq C_{n,p,q'}\|f\|_{L^{q'}(\mathbb{R}^n)}\|g\|_{L^p(\mathbb{R}^n)}.
\end{equation}
If $p=q'=\frac{2n}{2n-\lambda}$, then
$$C_{n,p,q'}=\pi^{\frac{\lambda}{n}}\frac{\Gamma(\frac{n}{2}-\lambda)}{\Gamma(n-\lambda)}\big(\frac{\Gamma(\frac{n}{2})}{\Gamma(n)}\big)^{\frac{\lambda-n}{n}}.$$
In this case, equality \eqref{EHL} holds if and only if
$$f=\frac{c_1}{\Big(d+|x-x_0|^2)^{\frac{2n-\lambda}{2}}}$$
for some $x_0\in \mathbb{R}^n$ and $d>0$.
\medskip

We mention that the Stein-Weiss inequalities with fractional Poisson kernels have been established recently by Chen, Liu, Lu and Tao \cite{CLLT} using the Hardy-Littlewood-Sobolev inequalities proved by Chen, Lu and Tao \cite{CLT}.

Another natural question is whether there exist some similar inequalities such as \eqref{SW} and \eqref{HL1} on the Heisenberg group?
Folland and Stein first in \cite{FS} give a positive answer in terms of fractional integral operator. For simplicity, we introduce some background knowledge about Heisenberg group.
The $n$-dimensional Heisenberg group $\mathbb{H}^n=\mathbb{C}^n\times \mathbb{R}$ is a Lie group with group structure given by
$$uv=(z,t)(z',t')=(z+z',t+t'+2Im(z\cdot\overline{z'}))$$
for any two points $u=(z,t)$ and $v=(z',t')$ in $\mathbb{H}^n$. Haar measure on $\mathbb{H}^n$ is the usual Lebesgue measure $du=dzdt$.
The Lie algebra of $\mathbb{H}^n$ is generated by the left invariant vector fields
$$T=\frac{\partial}{\partial t},\ \ X_j=\frac{\partial}{\partial x_j}+2y_j\frac{\partial}{\partial t},\ \ Y_j=\frac{\partial}{\partial y_j}+2x_j\frac{\partial}{\partial t}.$$
For each real number $\lambda \in \mathbb{R}$ and $u=(z,t) \in \mathbb{H}^n$, we denote the dilation $\delta_\lambda(u)=(\lambda z,\lambda^2t)$, the homogenous norm on $\mathbb{H}^n$ as $|u|=(|z|^4+t^2)^{\frac{1}{4}}$ and $Q=2n+2$ as the homogenous dimension. With this norm, a Heisenberg ball centered at $u=(z,t)\in \mathbb{H}^n$ with radius $R$
is given by $B_{\mathbb{H}}(u, R)=\{v\in \mathbb{H}: d(u, v)<R\}$, where $v=(z',t')$, $u^{-1}=(-z,-t)$, and $d(u,v):=|u^{-1}v|=|v^{-1}u|$ denotes left-invariant quasi-metric on Heisenberg group.
\medskip

Folland and Stein first \cite{FS} established the following Hardy-Littlewood-Sobolev inequality on the Heisenberg group.
\vskip 0.1cm

\noindent\textbf{Theorem B.}
For $1<q'$, $p<\infty$, $0<\lambda<Q$ and $\frac{1}{q'}+\frac{1}{p}+\frac{\lambda}{Q}=2$, there exists some constant $C_{Q,p,q'}$ such that
for any $f\in L^{q'}(\mathbb{H}^n)$, $g \in L^{p}(\mathbb{H}^n)$, there holds
\begin{equation}\label{HHL}
\int_{\mathbb{H}^n}\int_{\mathbb{H}^n} f(u)|u^{-1}v|^{-\lambda}g(v)dudv\leq C_{Q,p,q'}\|f\|_{L^{q'}(\mathbb{H}^n)}\|g\|_{L^p(\mathbb{H}^n)}.
\end{equation}
\vskip 0.1cm
In conjunction with the CR Yamabe problem on the CR manifolds (see \cite{JL1,JL2}), Jerison and Lee \cite{JL2} proved
the sharp version and gave the optimizer of the  $L^2(\mathbb{H}^n)$ to $L^{\frac{2Q}{Q-2}}(\mathbb{H}^n)$ Sobolev inequality on the Heisenberg group which is equivalent to the Hardy-Littlewood-Sobolev inequality  \eqref{HHL} for $\lambda=Q-2$ and $p=q'=\frac{2Q}{Q+2}$. Frank and Lieb \cite{FL3} established the best constants and extremal functions for this inequality  \eqref{HHL}
in the case of $p=q'=\frac{2Q}{2Q-\lambda}$. More precisely, they obtained the following result.

\noindent\textbf{Theorem C.}
For $0<\lambda<Q$ and $p=q'=\frac{2Q}{2Q-\lambda}$, then for any  $f\in L^{q'}(\mathbb{H}^n)$ and $g\in L^p(\mathbb{H}^n)$, there holds
\begin{equation}\label{EHL2}
\int_{\mathbb{H}^n}\int_{\mathbb{H}^n} f(u)|u^{-1}v|^{-\lambda}g(v)dudv\leq\Big(\frac{\pi^{n+1}}{2^{n-1}n!}\Big)^{\frac{\lambda}{Q}}\frac{n!\Gamma(\frac{Q-\lambda}{2})}{\Gamma^2(\frac{2Q-\lambda}{4})}\|f\|_{L^{q'}(\mathbb{H}^n)}\|g\|_{L^p(\mathbb{H}^n)}
\end{equation}
with equality if and only if
$$f(u)=c_1H(\delta(a^{-1}u)),\ \ \ g(u)=c_2H(\delta(a^{-1}u))$$ for some $c_1$, $c_2\in \mathbb{R}$, $\delta>0$ and $a\in \mathbb{H}^n$. Here $H$ is the function given by $H(u)=\Big((1+|z|^2)^2+t^2\Big)^{-\frac{2Q-\lambda}{4}}$.
\medskip

We remark as a borderline case of the sharp Sobolev inequality on the Heisenberg group, the sharp Trudinger-Moser inequality on finite domains was established by Cohn and Lu \cite{CL} and by Lam and Lu \cite{LL} for critical Trudinger-Moser inequality and by Lam, Lu and Tang \cite{LLT} for the subcritical Trudinger-Moser inequality on the entire Heisenberg group.

In the Euclidean space, by using the competing symmetry method, Carlen and Loss \cite{CarlenLoss}  provided a different proof from Lieb's of the sharp constants and extremal functions in the diagonal case $p=q^{'}=\frac{2n}{2n-\lambda}$ and Frank and Lieb \cite{FL1} offered a new proof using the reflection positivity of inversions in spheres in the special diagonal case. Carlen, Carillo and Loss gave a simple proof of the  sharp Hardy-Littlewood-Sobolev
inequality when $\lambda=n-2$ for $n\geq 3$ via a monotone
flow governed by the fast diffusion equation \cite{CCL}.
Frank and Lieb \cite{FL2} further employed a rearrangement-free technique developed in \cite{FL3} to recapture  the best constant of inequality \eqref{HHL}. This method has also been successfully applied to obtain sharp constants for similar inequalities on quaternionic and octonionic Heisenberg groups (see  Christ, Liu and Zhang \cite{Chliu1,Chliu2}). For Hardy-Littlewood-Sobolev inequality \eqref{HHL} on the Heisenberg group,  little was known about maximizers and sharp constants except the case $p=q'=\frac{2Q}{2Q-\lambda}$. Han \cite{Han0} proved the existence of extremals of inequality \eqref{HHL} for all the case of $q'$ and $p$ through the concentration compactness principle of Lions \cite{Lions1,Lions2}.
\medskip

Han, Lu and Zhu in \cite{HLZ} utilized the theory of weighted inequalities for integral operators to establish
the Stein-Weiss inequalities \eqref{SW} on the Heisenberg group.

\medskip

\noindent\textbf{Theorem D.}
For $1<p<\infty$, $1< q'<\infty$, $0<\lambda<Q=2n+2$ and $\alpha+\beta\geq0$ such that $\lambda+\alpha+\beta\leq Q$, $\beta<\frac{Q}{p'}$, $\alpha<\frac{Q}{q}$ and $\frac{1}{p}+\frac{1}{q'}+\frac{\lambda+\alpha+\beta}{Q}=2$,
there exists some constant $C_{Q,\alpha,\beta,p,q'}>0$ such that for any functions  $f\in L^{q'}(\mathbb{H}^n)$ and $g\in L^{p}(\mathbb{H}^n)$, there holds
\begin{equation}\label{HSW}
\int_{\mathbb{H}^n}\int_{\mathbb{H}^n}\frac{f(u)g(v)}{|u|^{\alpha}|u^{-1}v|^{\lambda}|v|^{\beta}}dvdu\leq C_{Q,\alpha,\beta,p,q'}\|f\|_{L^{q'}(\mathbb{H}^n)}\|g\|_{L^{p}(\mathbb{H}^n)}.
\end{equation}

Furthermore, they also studied the weighted Hardy-Littlewood-Sobolev inequalities with different weights, i.e., $|z|$ weights. More precisely, they obtained
\medskip

\noindent\textbf{Theorem E.}
For $1<p<\infty$, $1< q'<\infty$, $0<\lambda<Q=2n+2$ and $\alpha+\beta\geq0$ such that $\lambda+\alpha+\beta\leq Q$, $\beta<\frac{2n}{p'}$, $\alpha<\frac{2n}{q}$ and $\frac{1}{p}+\frac{1}{q'}+\frac{\lambda+\alpha+\beta}{Q}=2$,
there exists some constant $C_{Q,\alpha,\beta,p,q'}>0$ such that for any functions  $f\in L^{q'}(\mathbb{H}^n)$ and $g\in L^{p}(\mathbb{H}^n)$, there holds
\begin{equation}\label{HSW1}
\int_{\mathbb{H}^n}\int_{\mathbb{H}^n}\frac{f(u)g(v)}{|z|^{\alpha}|u^{-1}v|^{\lambda}|z'|^{\beta}}dvdu\leq C_{Q,\alpha,\beta,p,q'}\|f\|_{L^{q'}(\mathbb{H}^n)}\|g\|_{L^{p}(\mathbb{H}^n)},
\end{equation}
where $u=(z,t)$ and $v=(z', t')$.
\medskip

It is natural to ask whether there exist extremal functions for the Stein-Weiss inequalities. Due to the absence of Riesz rearrangement inequalities on the Heisenberg group,
one cannot simply follow the same line of Lieb \cite{Lieb} to establish the existence of extremals for the Stein-Weiss inequality on the Heisenberg group. On the other hand, it is well known that the concentration compactness principle is an essential tool in dealing with the existence of extremal functions for  geometrical inequalities. Han \cite{Han0} applied the concentration compactness argument, which was originally used by Lions to solve the existence of extremals for Sobolev inequalities in Euclidean spaces, to obtain the existence of extremals for the Hardy-Littlewood-Sobolev inequality on the Heisenberg group. For the Stein-weiss inequality, this problem is highly non-trivial because one cannot apply the same method of Han \cite{Han0} due to the loss of translation invariance of the Stein-Weiss inequality. We overcome this difficulty by using the method of combining the concentration compactness argument and
dilation invariance to establish the attainability of the best constant of the Stein-Weiss inequality \eqref{HSW} on the Heisenberg group.
For the above Stein-Weiss inequality with $|z|$ weights on the Heisenberg group, Beckner \cite{B1} gave the sharp constant when $\alpha=\beta=\frac{Q-\lambda}{2}$, $p=q=2$ and proved the nonexistence of optimizers. However, it is completely open in other general cases including existence question. We verify that there exist extremal functions for the $|z|$ weighted Stein-Weiss inequalities in Theorem \ref{thm3} in general case.
Define the operator
$$I_{\lambda}(g)(u)=\int_{\mathbb{H}^n}\frac{g(v)}{|u^{-1}v|^{\lambda}}dv.$$
By duality, we can see that inequality \eqref{HSW} is equivalent to the following doubly weighted inequality
\begin{equation}
\|I_{\lambda}(g)|u|^{-\alpha}\|_{L^{q}(\mathbb{H}^n)}\leq C_{n,\alpha,\beta,p,q'}\|g|v|^{\beta}\|_{L^{p}(\mathbb{H}^n)}.
\end{equation}
Consider the following maximizing problem
\begin{equation}\label{mini}
C_{Q,\alpha,\beta,p,q'}:= \sup \{\|I_{\lambda}(g)|u|^{-\alpha}\|_{L^q(\mathbb{H}^n)} : g\geq 0 , \|g|v|^{\beta}\|_{L^p(\mathbb{H}^{n})}=1\},
\end{equation}
it is easy to verify that the extremals for the Stein-Weiss inequality on the Heisenberg group are those solving the maximizing problem \eqref{mini}.
\begin{theorem}\label{thm1}
Under assumptions of Theorem D, if we furthermore assume that $q>p$, then
there exists some nonnegative function $g$ satisfying $\|g|v|^{\beta}\|_{L^p(\mathbb{H}^{n})}=1$ and $\|I_{\lambda}(g)|u|^{-\alpha}\|_{L^q(\mathbb{H}^n)}=C_{Q,\alpha,\beta, p,q'}$.
\end{theorem}
\begin{remark}
Lieb \cite{Lieb} employed the Riesz rearrangement inequalities to establish the existence of extremals for inequality \eqref{SW}. Hence they must assume that
both $\alpha$ and $\beta$ are nonnegative. In our proof, we remove this assumption. Our method can be applied to Euclidean space without any change. Hence, our existence results of extremal functions are also new in $\mathbb{R}^{n}$. We state our new result in Euclidean space as follows and its proof can be given identically as on the Hiesenberg group and therefore we needn't give a separate proof.

\end{remark}

\begin{theorem}
For $1<p<\infty$, $1< q'<\infty$, $0<\lambda<n$ and $\alpha+\beta\geq0$ such that $\lambda+\alpha+\beta\leq n$, $\beta<\frac{n}{p'}$, $\alpha<\frac{n}{q}$ and $\frac{1}{p}+\frac{1}{q'}+\frac{\lambda+\alpha+\beta}{n}=2$,
there exists some constant $C_{n,\alpha,\beta,p,q'}>0$ such that for any functions  $f\in L^{q'}(\mathbb{R}^n)$ and $g\in L^{p}(\mathbb{R}^n)$, there holds
\begin{equation}\label{ESW}
\int_{\mathbb{R}^n}\int_{\mathbb{R}^n}\frac{f(x)g(y)}{|x|^{\alpha}|x-y|^{\lambda}|y|^{\beta}}dxdy\leq C_{n,\alpha,\beta,p,q'}\|f\|_{L^{q'}(\mathbb{R}^n)}\|g\|_{L^{p}(\mathbb{R}^n)}.
\end{equation}
Furthermore, if we assume that $q>p$, then the best constant $C_{n,\alpha,\beta,p,q'}$ could be achieved.
\end{theorem}
Similar to Theorem \ref{thm1}, we also obtain the existence result involved with the Stein-Weiss inequalities with $|z|$ weights on the Heisenberg group.
\begin{theorem}\label{thm3}
Under the assumptions of Theorem E, if we furthermore assume that $q>p$, then
there exists some nonnegative function $g$ satisfying $\|g|z'|^{\beta}\|_{L^p(\mathbb{H}^{n})}=1$ and $\|I_{\lambda}(g)|z|^{-\alpha}\|_{L^q(\mathbb{H}^n)}=C_{Q,\alpha,\beta, p,q'}$.
\end{theorem}
Next, we will prove a more general result.

\begin{theorem}\label{thm5}
For $1<p<\infty$, $1< q'<\infty$, $0<\lambda<n$ and $\alpha+\beta\geq0$ such that $\lambda+\alpha+\beta\leq n+m$, $\beta<\frac{m}{p'}$, $\alpha<\frac{m}{q}$ and $\frac{1}{p}+\frac{1}{q'}+\frac{\lambda+\alpha+\beta}{n+m}=2$,
there exists some constant $C_{n,m,\alpha,\beta,p,q'}>0$ such that for any functions  $f\in L^{q'}(\mathbb{R}^{n+m})$ and $g\in L^{p}(\mathbb{R}^{n+m})$, there holds
\begin{equation}\label{HSWF}
\int_{\mathbb{R}^{n+m}}\int_{\mathbb{R}^{n+m}}\frac{f(x)g(y)}{|x'|^{\alpha}|x-y|^{\lambda}|y'|^{\beta}}dxdy\leq C_{n,m,\alpha,\beta,p,q'}\|f\|_{L^{q'}(\mathbb{R}^{n+m})}\|g\|_{L^{p}(\mathbb{R}^{n+m})},
\end{equation}
where $x=(x', x'')$, $y=(y', y'')\in \mathbb{R}^m\times \mathbb{R}^n$.
Furthermore, if we assume that $q>p$, then the best constant $C_{n,\alpha,\beta,p,q'}$ could be achieved.
\end{theorem}

Once we establish the existence of extremals for the inequality \eqref{HSWF}, we naturally are concerned about some properties such as radial symmetry for their extremals. In order to achieve this purpose, one can maximize the functional
\begin{equation}\label{functional}
J(f,g)=\int_{\mathbb{R}^{n+m}}\int_{\mathbb{R}^{n+m}}\frac{f(x)g(y)}{|x'|^{\alpha}|x-y|^{\lambda}|y'|^{\beta}}dxdy
\end{equation}
under the constraint $\|f\|_{L^{q'}}=\|g\|_{L^p}=1$. Assume that $f$ and $g$ is a pair of extremal functions to the inequality \eqref{HSWF},
According to Euler-Lagrange multipliers theorem, we obtain
\begin{equation}\label{eu}\begin{cases}
J(f,g)f(x)^{q'-1}=\int_{\mathbb{R}^{n+m}}|x'|^{-\alpha}|x-y|^{-\lambda} g(y)|y'|^{-\beta} dy,\\
J(f,g)g(x)^{p-1}=\int_{\mathbb{R}^{n+m}}|x'|^{-\beta}|x-y|^{-\lambda} f(y)|y'|^{-\alpha} dy.
\end{cases}\end{equation}
\medskip
Set $u=c_1f^{q'-1}, v=c_2g^{p-1}, \frac{1}{q'-1}=p_1$ and $\frac{1}{p-1}=p_2$, then for a proper choice of
constants $c_1$ and $c_2$, system \eqref{eu} becomes
\begin{equation}\label{system3}\begin{cases}
u(x)=\int_{\mathbb{R}^{n+m}}|x'|^{-\alpha}|x-y|^{-\lambda} v^{p_2}(y)|y'|^{-\beta} dy,\\
v(x)=\int_{\mathbb{R}^{n+m}}|x'|^{-\beta}|x-y|^{-\lambda} u^{p_1}(y)|y|^{-\alpha} dy,
\end{cases}\end{equation}
where $\frac{1}{p_1-1}+\frac{1}{p_2-1}=\frac{\alpha+\beta+\lambda}{n+m}.$
Applying the methods of moving plane in integral forms \cite{CLO}, we obtain symmetry results for positive solutions of the system \eqref{system3}.

\begin{theorem}\label{thm6}
Assume that $(u,v)\in L^{p_1+1}(\mathbb{R}^{m+n})\times L^{p_2+1}(\mathbb{R}^{m+n})$ is a pair of positive solutions of the integral system \eqref{system3}, then $u(x)|_{\mathbb{R}^m\times \{0\}}$ and $v(x)|_{\mathbb{R}^m\times \{0\}}$ are radially symmetric and monotone decreasing about the origin in $\mathbb{R}^m$, $u(x)|_{\{0\}\times \mathbb{R}^n}$ and $v(x)|_{\{0\} \times \mathbb{R}^n}$ are radially symmetric and monotone decreasing about some $x_0 \in \mathbb{R}^n$.
\end{theorem}

\medskip

\section{The Proof of Theorem \ref{thm1}}
In this section, we will give the proof of attainability of the best constant for the inequality \eqref{HSW}. Though our proof is inspired by  Lions' proof of the existence of extremals for the classical Sobolev inequalities,  the loss of translation invariance  for the Stein-Weiss inequalities (namely, doubly weighted Hardy-Littlewood-Sobolev inequalities) presents additional difficulty because of the presence of weights. Moreover, lack of the analogous  compact imbedding associated with the integral inequalities produces more difficulties in applying directly the Lions concentration compactness principle.

\medskip

Assume that $\{\tilde{g}_{i}\}_i$ is a maximizing sequence for problem \eqref{mini}, namely $\|\tilde{g_i}|v|^{\beta}\|_{L^p(\mathbb{H}^n)}=1$ and
$\lim\limits_{i\rightarrow +\infty}\|I_{\lambda}(\tilde{g_i})|u|^{-\alpha}\|_{L^q(\mathbb{H}^n)}=C_{Q,\alpha,\beta,p,q'}$.
Since $\int_{\mathbb{H}^n}|\tilde{g_i}|^p|v|^{\beta p}dv=1$ for all $i\in \mathbb{N}$, there exists $r_i>0$ such that
$$\int_{B_\mathbb{H}(0, r_i)}|\tilde{g_i}|^p|v|^{\beta p}dv=\frac{1}{2}.$$  We define a new sequence $\{g_i\}_i$ given by
$$g_i(v):=r_i^{(Q+\beta p)/p}\widetilde{g}_i(r_iv)$$
for all $i\in \mathbb{N}$ and $v=(z,t)\in \mathbb{H}^n$. Clearly $\{g_i\}_i$ is also a maximizing sequence for problem \eqref{mini}.
Moreover, we also have that
$$\int_{B_\mathbb{H}(0,1)}|g_i(v)|^p|v|^{\beta p}dv=\frac{1}{2}$$
for all $i\in \mathbb{N}$. We define the measure:
$$\mu_i:=|g_i|^p|v|^{\beta p}dv \ \ and\ \ \nu_i:=|I_{\lambda}(g_i)|^q|u|^{-\alpha q}du.$$
Then it follows that there exist two bounded measures $\mu$ and $\nu$ such that
$$\mu_i\rightharpoonup \mu \ \ and\ \ \ \nu_i\rightharpoonup \nu \ \ weakly \ in\  the\ sense \ of\  measures\  as\  i\rightarrow\infty.$$
Furthermore, by the lower semi-continuity in the sense of measure, one can get
$$\int_{\mathbb{H}^n}d\mu\leq \lim_{i\rightarrow +\infty}\int_{\mathbb{H}^n}d\mu_i=1\ \ {\rm and}\ \ \int_{\mathbb{H}^n}d\nu\leq \lim_{i\rightarrow +\infty}\int_{\mathbb{H}^n}d\nu_i=C^q_{Q,\alpha,\beta,p,q'}.$$
We now apply Lions' first concentration compactness (see \cite{Lions1,Lions2}) to the sequence of measure $\{\mu_i\}$.
\medskip
\begin{lemma}
There exists a subsequence of $\{\mu_i\}$ such that one of the following condition holds:
(a)(Compactness) There exists $\{v_i\}$ in $\mathbb{H}^n$ such that for each $\epsilon>0$ small enough, we can find $R_{\epsilon}>0$
with $$\int_{B_\mathbb{H}(v_i,R_\varepsilon)}d\mu_i\geq1-\varepsilon \ \ {\rm for}\ {\rm all}\ i.$$
(b)(Vanishing) For all $R>0$, there holds:
$$\lim_{i\rightarrow\infty}(\sup_{v\in \mathbb{H}^n}\int_{B_{\mathbb{H}}(v,R)}d\mu_i)=0 .$$
(c)(Dichotomy)\ There exists $k \in(0,1)$ such that for any $\varepsilon >0$, there exist $R_\varepsilon>0$ and a sequence $\{v_i^\varepsilon\}_{i\in \mathbb{N}}$, with the following property: given $R'>R_\varepsilon$, there are non-negative measures $\mu_i^1$ and $\mu_i^2$ such that
$$0\leq \mu_i^1+\mu_i^2\leq \mu_i,\ \ Supp(\mu_i^1)\subset  B_{\mathbb{H}}(v_i^\varepsilon,R_\varepsilon),
\ \  Supp(\mu_i^2)\subset \mathbb{H}^n\setminus B_{\mathbb{H}}(v_i^\varepsilon,R'),$$
$$ \mu_i^1=\mu_i|_{B_{\mathbb{H}}(v_i^\varepsilon,R_\varepsilon)},
\ \ \mu_i^2= \mu_i|_{\mathbb{H}^n\setminus B_{\mathbb{H}}(v_i^\varepsilon,R')},$$
$$\limsup_{i\rightarrow\infty}\Big(|k- \int_{\mathbb{H}^n}d\mu_i^1|+|(1-k)-\int_{\mathbb{H}^n}d\mu_i^2|\Big)\leq \varepsilon.$$
\end{lemma}

\begin{lemma}\label{lem3}
Let $\{g_i\}_i$ is a maximizing sequence for problem \eqref{mini}, then compactness (a) holds. In particular, we also have that $\int_{\mathbb{H}^n}d\mu=1$.
\end{lemma}
\begin{proof}
Clearly, (b) can not occur because $\int_{B_\mathbb{H}(0,1)}|g_i(v)|^p|v|^{\beta p}dv=\frac{1}{2}$. Next, it suffices to show that (c) also does not happen.
We argue this by contradiction. Suppose that (c) occurs, then for any $\varepsilon>0$, we can find $R_{0}>0$ and $\{v_i\}\subset \mathbb{H}^n$ such that for any given $R\geq R_0$, there holds
\begin{equation}\label{c1}
\|g_i|v|^{\beta}\chi_{B_\mathbb{H}(v_i,R)}\|_p^p=k+ O(\varepsilon)\ \ {\rm and}\ \ \|g_i|v|^{\beta}\chi_{B_\mathbb{H}^c(v_i,R)}\|_p^p=1-k+ O(\varepsilon).
\end{equation}
We first claim that
$$\lim_{R\rightarrow \infty}I_{\lambda}(g_i\chi_{B_\mathbb{H}(v_i,R)})(u)=I_{\lambda}(g_i)(u)\ \ {\rm for}\ u\in \mathbb{H}^n.$$
In fact,
\begin{equation}\begin{split}
&|I_{\lambda}(g_i)(u)-I_{\lambda}(g_i\chi_{B_\mathbb{H}(v_i,R)})(u)|\\
&\ \ =|I_{\lambda}(g_i\chi_{B_\mathbb{H}^c(v_i,R)})(u)|\\
&\ \ =|\int_{\mathbb{H}^n}\frac{(g_i\chi_{B_\mathbb{H}^c(v_i,R)})(v)}{|u^{-1}v|^{\lambda}}dv|\\
&\ \ \leq \big(\int_{B_\mathbb{H}^c(v_i,R)}|g_i(v)|v|^{\beta}|^pdv\big)^{\frac{1}{p}}\big(\int_{B_\mathbb{H}^c(v_i,R)}|u^{-1}v|^{-\lambda p'}|v|^{-\beta p'}dv\big)^{\frac{1}{p'}}\\
&\ \ \rightarrow 0
\end{split}\end{equation}
as $R\rightarrow +\infty$. Now we can apply the Brezis-Lieb lemma (\cite{BL}) to obtain
$$\|I_{\lambda}(g_i)|u|^{-\alpha}\|_{q}^q=\|I_{\lambda}(g_i\chi_{B_\mathbb{H}(v_i,R)})|u|^{-\alpha}\|_{q}^q+
\|I_{\lambda}(g_i\chi_{B_\mathbb{H}^c(v_i,R)})|u|^{-\alpha}\|_{q}^q+o(1).$$
In light of the $|u|$ weighted Stein-Weiss inequality \eqref{HSW} and condition \eqref{c1}, we derive that
\begin{equation}\begin{split}
&\|I_{\lambda}(g_i\chi_{B_\mathbb{H}(v_i,R)})|u|^{-\alpha}\|_{q}^q+
\|I_{\lambda}(g_i\chi_{B_\mathbb{H}^c(v_i,R)})|u|^{-\alpha}\|_{q}^q+o(1)\\
&\ \ \leq C_{Q,\alpha,\beta, p,q'}^q\|g_i\chi_{B_\mathbb{H}(v_i,R)}|v|^{\beta}\|_p^q+C_{Q,\alpha,\beta, p,q'}^q\|g_i\chi_{B_\mathbb{H}^c(v_i,R)}|v|^{\beta}\|_p^q+o(1)\\
&\ \ \leq C_{Q,\alpha,\beta, p,q'}^q(k+O(\varepsilon))^{\frac{q}{p}}\leq C_{Q,\alpha,\beta, p,q'}^q(1-k+O(\varepsilon))^{\frac{q}{p}}+o(1)\\
&\ \ \leq C_{Q,\alpha,\beta, p,q'}^q(k^{\frac{q}{p}}+(1-k)^{\frac{q}{p}})+O(\varepsilon)+o(1)\\
&\ \ < C_{Q,\alpha,\beta, p,q'}^q,
\end{split}\end{equation}
which is a contradiction with $\|I_{\lambda}(g_i)|u|^{-\alpha}\|_{q}^q\rightarrow C_{Q,\alpha,\beta, p,q'}^q$. Hence compactness condition (a) must hold.
\medskip

Now, we turn to the proof of $\int_{\mathbb{H}^n}d\mu=1$. It suffices to show that $\int_{\mathbb{H}^n}d\mu \geq 1$.
 According to $\int_{B_\mathbb{H}(0,1)}|g_i|^p|v|^{\beta p}dv=\frac{1}{2}$, we claim that
$\{v_i\}$ must be a bounded sequence in $\mathbb{H}^n$. We argue this by contradiction. If $\{v_i\}$ is unbounded, then we can find subsequence
$v_i$ satisfying $|v_i|\rightarrow \infty$. In view of $(a)$, we obtain for any sufficiently small $\epsilon>0$, there exist
$R_\epsilon>0$ such that $$\int_{B_\mathbb{H}^c(v_i,R_\epsilon)}|g_i|^p|v|^{\beta p}dv\leq \epsilon.$$ On the other hand,
for any $v \in B_{\mathbb{H}}(0,1)$, there holds $|v-v_i|\geq |v_i|-|v|\rightarrow \infty$. Then it follows
that for sufficiently large $i$, there holds $$\int_{B_\mathbb{H}(0,1)}|g_i|^p|v|^{\beta p}dv\leq
\int_{B_\mathbb{H}^c(v_i,R_\epsilon)}|g_i|^p|v|^{\beta p}dv\leq \epsilon,$$
which is a contradiction. Then $v_i$ must be a bounded sequence in $\mathbb{H}^n$. Hence, we may assume for any $\varepsilon>0$, there exist $R_0$ such that for any $R\geq R_0$, there holds
$$\int_{B_\mathbb{H}(0,R)}d\mu_i\geq1-\varepsilon.$$
Then for any $R>R_0$ and $\phi \in C^{\infty}_c(\mathbb{H}^n)$ with $\phi|_{B_\mathbb{H}(0,R)}=1$, then by the weak convergence $\mu_i$ in the sense of measure, we obtain
\begin{equation}\begin{split}
\int_{\mathbb{H}^n}d\mu&\geq \int_{\mathbb{H}^n}\phi(v)d\mu\\
&=\lim_{i\rightarrow \infty}\int_{\mathbb{H}^n}|g_i|^p|v|^{\beta p}\phi(v)dv\\
&\geq \lim_{i\rightarrow \infty}\int_{B_\mathbb{H}(0,R)}|g_i|^p|v|^{\beta p}dv\\
&\geq \lim_{i\rightarrow \infty}\int_{B_\mathbb{H}(0,R_0)}d\mu_i
\geq 1-\varepsilon.\\
\end{split}\end{equation}
Let $\varepsilon\rightarrow 0$, we derive that $\int_{\mathbb{H}^n}d\mu \geq 1$. Then we accomplish the proof of Lemma \ref{lem3}.
\end{proof}

\begin{lemma}\label{lem4}
Let $\{g_{i}\}_i$ be a maximizing sequence satisfying
\begin{equation}\label{t1}
\|g_{i}(v)|v|^{\beta}\|_{L^{p}(\mathbb{H}^n)}=1\ \ {\rm and}\ \ \int_{B_{\mathbb{H}}(0,R)}|g_{i}(v)|^{p}|v|^{\beta p}dv\geq1-\varepsilon(R).
\end{equation}
we may assume that $g_{i}(v)|v|^{\beta}\rightarrow g(v)|v|^{\beta}$ weakly in $L^{p}(\mathbb{H}^n)$ (by passing to a subsequence if necessary). Then, by passing to a subsequence again  if necessary,
$$I_{\lambda}(g_{i})(u)|u|^{-\alpha}\rightarrow I_{\lambda}(g)(u)|u|^{-\alpha}\ \  a.e..$$
\end{lemma}
\begin{proof}
We show that $I_{\lambda}(g_{i})(u)|u|^{-\alpha}\rightarrow I_{\lambda}(g)(u)|u|^{-\alpha}$ in measure to ensure the existence of a pointwisely convergent subsequence of $\{g_{i}\}$. Observe that for $M>R$, it follows from \eqref{HSW} and \eqref{t1} that
\begin{equation}\begin{split}
&\|I_{\lambda}(g_{i})(u)|u|^{-\alpha}\chi_{B_{\mathbb{H}}^c(0,M)}\|_{L^{q}(\mathbb{H}^n)}\\
\leq &\|I_{\lambda}(g_{i}\chi_{B_{\mathbb{H}}(0,R)})(u)|u|^{-\alpha}\chi_{B_{\mathbb{H}}^{c}(0,M)}\|_{L^{q}(\mathbb{H}^n)}
+\|I_{\lambda}(g_{i}\chi_{B_{\mathbb{H}}^{c}(0,R)})(u)|u|^{-\alpha}\chi_{B_{\mathbb{H}}^{c}(0,M)}\|_{L^{q}(\mathbb{H}^n)}\\
\leq &\|I_{\lambda}(g_{i}\chi_{B_{\mathbb{H}}(0,R)})(u)|u|^{-\alpha}\chi_{B_{\mathbb{H}}^{c}(0,M)}\|_{L^{q}(\mathbb{H}^n)}
+C_{n,\alpha,\beta,p,q'}\|(g_{i}\chi_{B_{\mathbb{H}}^{c}(0,R)})(v)|v|^{\beta}\|_{L^{p}(\mathbb{H}^n)}\\
\leq &\|I_{\lambda}(g_{i}\chi_{B_{\mathbb{H}}(0,R)})(u)|u|^{-\alpha}\chi_{B_{\mathbb{H}}^{c}(0,M)}\|_{L^{q}(\mathbb{H}^n)}
+\epsilon(R).
\end{split}\end{equation}

Since $\frac{1}{p}+\frac{1}{q'}+\frac{\lambda+\alpha+\beta}{Q}=2$ and $\beta<\frac{Q}{p'}$,  we get $Q<(\alpha+\lambda)q$ and
$$\|g_{i}\chi_{B_{\mathbb{H}}(0,R)}\|_{L^{1}(\mathbb{H}^n)}\leq\|g_{i}(v)|v|^{\beta}\|_{L^{p}(\mathbb{H}^n)}\|\frac{1}{|v|^{\beta}}\chi_{B_{\mathbb{H}}(0,R)}\|_{L^{p'}(\mathbb{H}^n)}<\infty.$$
Using Minkowski's inequality, we obtain
\begin{equation}\begin{split}
&\|I_{\lambda}(g_{i}\chi_{B_{\mathbb{H}}(0,R)})(u)|u|^{-\alpha}\chi_{B_{\mathbb{H}}^{c}(0,M)}\|_{L^{q}(\mathbb{H}^n)}\\
\leq&\|g_{i}\chi_{B_{\mathbb{H}}(0,R)}\|_{L^{1}(\mathbb{H}^n)}\left(\int_{|u|\geq M}\frac{1}{(|u|-R)^{\lambda q}}|u|^{-\alpha q}du\right)^{\frac{1}{q}}\\
\lesssim&\|g_{i}|v|^{\beta}\|_{L^{p}_{B_{\mathbb{H}}(0,R)}}\||v|^{-\beta}\|_{L^{p'}_{B_{\mathbb{H}}(0,R)}}\frac{1}{M^{\alpha+\lambda-\frac{Q}{q}}}\rightarrow 0
\end{split}\end{equation}
for every fixed $R$ as $M\rightarrow\infty$. Therefore, we have
\begin{equation}\label{tight}
\|I_{\lambda}(g_{i})(u)|u|^{-\alpha}\chi_{B_{\mathbb{H}}^{c}(0,M)}\|_{L^{q}(\mathbb{H}^n)}\leq\epsilon(M)
\end{equation}
and that is,
$$\|I_{\lambda}(g_{i})(u)|u|^{-\alpha}-I_{\lambda}(g_{i})(u)|u|^{-\alpha}\chi_{B_{\mathbb{H}}(0,M)}\|_{L^{q}(\mathbb{H}^n)}\leq\epsilon(M).$$
Since $g_{i}(v)|v|^{\beta}\rightarrow g(v)|v|^{\beta}$ weakly in $L^{p}(\mathbb{H}^n)$, we have
$$\|g(v)|v|^{\beta}\chi_{B_{\mathbb{H}}^{c}(0,R)}\|^{p}_{L^{p}(\mathbb{H}^n)}\leq\liminf_{j\rightarrow\infty}\|g_{i}(v)|v|^{\beta}\chi_{B_{\mathbb{H}}^{c}(0,R)}\|^{p}_{L^{p}(\mathbb{H}^n)}\leq\epsilon(R).$$
Similarly, for $g$ one can derive,
$$\|I_{\lambda}(g)(u)|u|^{-\alpha}-I_{\lambda}(g)(u)|u|^{-\alpha}\chi_{B_{\mathbb{H}}(0,M)}\|_{L^{q}(\mathbb{H}^n)}\leq\epsilon(M).$$
Therefore, given $k>0$,
\begin{equation}\begin{split}
&|\{|I_{\lambda}(g_{i})(u)|u|^{-\alpha}-I_{\lambda}(g)(u)|u|^{-\alpha}|\geq15k\}|\\
\leq&|\{|I_{\lambda}(g_{i})(u)|u|^{-\alpha}-I_{\lambda}(g_{i})(u)|u|^{-\alpha}\chi_{B_{M(0)}}(u)|\geq 5k\}|+\\
&|\{|I_{\lambda}(g_{i})(u)|u|^{-\alpha}\chi_{B_{\mathbb{H}}(0,M)}(u)-I_{\lambda}(g)(u)|u|^{-\alpha}\chi_{B_{\mathbb{H}}(0,M)}(u)|\geq 5k\}|+\\
&|\{|I_{\lambda}(g)(u)|u|^{-\alpha}\chi_{B_{\mathbb{H}}(0,M)}(u)-I_{\lambda}(g)(u)(u)|u|^{-\alpha}|\geq 5k\}|\\
\leq&2\big(\frac{\epsilon(M)}{5K}\big)^{q}+|\{|I_{\lambda}(g_{i})(u)|u|^{-\alpha}-I_{\lambda}(g)(u)|u|^{-\alpha}|\geq5k\}\cap B_{\mathbb{H}}(0,M)|\\
\end{split}\end{equation}
Thus, it remains to estimate the second term above. Denote
$$I_{\lambda}^{\eta}(g)(u)=\int_{B_{\mathbb{H}}^{c}(u,\eta)}\frac{g(v)}{|u^{-1}v|^{\lambda}}dv,$$
then
$$I_{\lambda}^{\eta}(g_{i}\chi_{B_{\mathbb{H}}(0,R)})(u)|u|^{-\alpha}\rightarrow I_{\lambda}^{\eta}(g\chi_{B_{\mathbb{H}}(0,R)})(u)|u|^{-\alpha}$$
for all $u\in\mathbb{H}^n$ because
$$|u^{-1}v|^{-\lambda}\chi_{B_{\mathbb{H}}(0,R)}\chi_{B_{\mathbb{H}}^{c}(u,\eta)}\in L^{p'}(\mathbb{H}^n)$$
for any fixed $u\in\mathbb{H}^n$ and $\eta>0$. Therefore, $I_{\lambda}^{\eta}(g_{i}\chi_{B_{R}(0)})(u)|u|^{-\alpha}\rightarrow I_{\lambda}^{\eta}(g\chi_{B_{\mathbb{H}}(0,R)})(u)|u|^{-\alpha}$ locally in measure, which means,
\begin{equation}\label{bu1}
|\{|I^{\eta}_{\lambda}(g_{i}\chi_{B_{\mathbb{H}}(0,R)})(u)|u|^{-\alpha}-I^{\eta}_{\lambda}(g\chi_{B_{\mathbb{H}}(0,R)})(u)|u|^{-\alpha}|\geq k\}\cap B_{\mathbb{H}}(0,M)|=o(1).
\end{equation}
On the other hand, since $\lambda+\alpha+\beta\leq Q$, we can derive that
\begin{equation*}\begin{split}
&\|I_{\lambda}(g_{i}\chi_{B_{\mathbb{H}}(0,R)})(u)|u|^{-\alpha}-I^{\eta}_{\lambda}(g_{i}\chi_{B_{\mathbb{H}}(0,R)})(u)|u|^{-\alpha}\|_{L^{1}(\mathbb{H}^n)}\\
&\ \ =\int_{\mathbb{H}^n}\left|\int_{B_{\mathbb{H}}(u,\eta)}\frac{g_{i}(v)\chi_{B_{\mathbb{H}}(0,R)}}{|u^{-1}v|^{\lambda}|u|^{\alpha}}dv\right|du\\
&\ \ \leq \Big\|g_{i}(v)\chi_{B_{\mathbb{H}}(0,R)}\Big\|_{L^{1}(\mathbb{H}^n)}\cdot \Big\|\int_{B_{\mathbb{H}}(v,\eta)}\frac{1}{|u^{-1}v|^{\lambda}|u|^{\alpha}}du\Big\|_{L^{\infty}(\mathbb{H}^n)}\\
&\ \ \leq C {\eta}^{Q-\lambda-\alpha}\rightarrow 0
\end{split}\end{equation*}
for every fixed $R$ as $\eta\rightarrow0$, where we use the fact
\begin{equation}\begin{split}
\int_{B_{\mathbb{H}}(v,\eta)}\frac{1}{|u^{-1}v|^{\lambda}|u|^{\alpha}}du&\leq \int_{B_{\mathbb{H}}(v,\eta)\cap\{|u^{-1}v|\geq |u|\}}\frac{1}{|u|^{\lambda+\alpha}}+\int_{B_{\mathbb{H}}(v,\eta)\cap\{|u^{-1}v|\leq |u|\}}\frac{1}{|u^{-1}v|^{\lambda+\alpha}}\\
&\leq \int_{B_{\mathbb{H}}(0,\eta)}\frac{1}{|u|^{\lambda+\alpha}}du+\int_{B_{\mathbb{H}}(v,\eta)}\frac{1}{|u^{-1}v|^{\lambda+\alpha}}du.
\end{split}\end{equation}

That is,
\begin{equation}\label{bu2}
\|I_{\lambda}(g_{i}\chi_{B_{\mathbb{H}}(0,R)})|u|^{-\alpha}-I_{\lambda}^{\eta}(g_{i}\chi_{B_{\mathbb{H}}(0,R)})|u|^{-\alpha}\|_{L^{1}(\mathbb{H}^n)}\leq O(\eta).
\end{equation}
Similarly, we can derive the analogous statement for $f$,
\begin{equation}\label{bu3}
\|I_{\lambda}(g\chi_{B_{R}(0)})|u|^{-\alpha}-I_{\lambda}^{\eta}(g\chi_{B_{R}(0)})|u|^{-\alpha}\|_{L^{1}(\mathbb{H}^n)}\leq O(\eta).
\end{equation}
Also notice that
\begin{equation}\label{bu4}
\|I_{\lambda}(g_{i})|u|^{-\alpha}-I_{\lambda}(g_{i}\chi_{B_{\mathbb{H}}(0,R)})|u|^{-\alpha}\|_{L^{q}(\mathbb{H}^n)}\leq C_{n,\alpha,\beta,p,q'}\|g\chi_{B_{\mathbb{H}}^{c}(0,R)}\|_{L^{p}(\mathbb{H}^n)}\leq \epsilon(R)
\end{equation}
and
\begin{equation}\label{bu5}
\|I_{\lambda}(g_{i})|u|^{-\alpha}-I_{\lambda}(g_{i}\chi_{B_{\mathbb{H}}(0,R)})|u|^{-\alpha}\|_{L^{q}(\mathbb{H}^n)}\leq C_{n,\alpha,\beta,p,q'}\|g\chi_{B^{c}_{\mathbb{H}}(0,R)}\|_{L^{p}(\mathbb{H}^n)}\leq \epsilon(R)
\end{equation}
Combining \eqref{bu1}-\eqref{bu5}
\begin{equation}\begin{split}
&|\{|I_{\lambda}(g_{i})(u)|u|^{-\alpha}-I_{\lambda}(g)(u)|u|^{-\alpha}|\geq5k\}\cap B_{M(0)}|\\
\leq&|\{|I_{\lambda}(g_{i})(u)|u|^{-\alpha}-I_{\lambda}(g_{i}\chi_{B_{\mathbb{H}}(0,R)})(u)|u|^{-\alpha}|\geq k\}|+\\
&|\{|I_{\lambda}(g_{i}\chi_{B_{\mathbb{H}}(0,R)})(u)|u|^{-\alpha}-I^{\eta}_{\lambda}(g_{i}\chi_{B_{\mathbb{H}}(0,R)})(u)|u|^{-\alpha}|\geq k\}|+\\
&|\{|I^{\eta}_{\lambda}(g_{i}\chi_{B_{\mathbb{H}}(0,R)})(u)|u|^{-\alpha}-I^{\eta}_{\lambda}(g\chi_{B_{\mathbb{H}}(0,R)})(u)|u|^{-\alpha}|\geq k\}\cap B_{\mathbb{H}}(0,M)|+\\
&|\{|I^{\eta}_{\lambda}(g\chi_{B_{\mathbb{H}}(0,R)})(u)|u|^{-\alpha}-I_{\lambda}(g\chi_{B_{\mathbb{H}}(0,R)})(u)|u|^{-\alpha}\}|+\\
&|\{I_{\lambda}(g\chi_{B_{\mathbb{H}}(0,R)})(u)|u|^{-\alpha}-I_{\lambda}(g)(u)|u|^{-\alpha}|\geq k\}|\\
\leq&2\big(\frac{\epsilon(R)}{k}\big)^{q}+2\frac{O(\eta)}{k}+o(1).
\end{split}\end{equation}
\end{proof}

\begin{lemma}\label{lem51}
There exist $I\subset \mathbb{N}$ at most countable and a family $\{u_i\}_{i\in I}$ in $\mathbb{H}^n$ such that
$$\nu=|I_{\lambda}(g)|^q|u|^{-\alpha q}du+\sum_{i\in I}\nu^i \delta_{u_i}$$
and $$\mu \geq |g|^p|v|^{\beta p}dv+\sum_{i\in I}\mu^i \delta_{u_i}$$
where $\nu^i=\nu(u_i)$, $\mu^i=\mu(u_i)$ for all $i\in I$.
Furthermore, we also have $\nu^i\leq C_{Q,\alpha,\beta,p,q'}^q(\mu^i)^{\frac{q}{p}}$.
\end{lemma}
The original version of the above Lemma is on $\mathbb{R}^n$, but it can also be done on $\mathbb{H}^n$ because
the crucial ingredient as Lemma 1.2 in \cite{Lions1} still holds in $\mathbb{H}^n$. Thus we omit the detailed proof.
\medskip

\emph{The Proof of Theorem \ref{thm1}:}
 Assumptions of Theorem \ref{thm1} imply that one of $\alpha$ and $\beta$ is nonnegative. Without loss of generality, we may assume that
$\alpha>0$. Once we obtain $\|g(v)|v|^{\beta}\|_p=1$, in view of the $|u|$ weighted Stein-Weiss inequalities and $g_i(v)|v|^{\beta}\rightharpoonup g(v)|v|^{\beta}$ weakly in $L^p(\mathbb{H}^n)$, then $g_i(v)|v|^{\beta}\rightarrow g(v)|v|^{\beta}$ strongly in $L^p(\mathbb{H}^n)$ and $\|I_{\lambda}(g)|u|^{-\alpha}\|_q=C_{Q,\alpha,\beta,p,q'}$. Hence $g$ is actually a extremal function for the $|u|$ weighted Stein-Weiss inequalities on Heisenberg group. Now, we start to prove that $\|g(v)|v|^{\beta}\|_p=1$. We argue this by contradiction.
By Lemma \ref{lem3}, we have $\mu(\mathbb{H}^n)=1$. In view of \eqref{tight}, as what we did in Lemma \ref{lem3}, we can deduce that  $\nu(\mathbb{H}^n)=C^q_{Q,\alpha,\beta,p,q'}$.
In light of Lemma \ref{lem51}, we derive that
\begin{equation}\begin{split}
\nu(\mathbb{H}^n)&=\|I_{\lambda}(g)|u|^{-\alpha}\|_q^q+\sum_{i\in I}\nu^i\\
&\leq C^q_{Q,\alpha,\beta,p,q'}\|g(v)|v|^{\beta}\|_p^q+\sum_{i\in I} C^q_{Q,\alpha,\beta,p,q'}(\mu^i)^{\frac{q}{p}}\\
&\leq C^q_{Q,\alpha,\beta,p,q'}k^{\frac{q}{p}}+ C^q_{Q,\alpha,\beta,p,q'}(\sum_{i\in I}\mu^i)^{\frac{q}{p}}\\
&\leq C^q_{Q,\alpha,\beta,p,q'}k^{\frac{q}{p}}+ C^q_{Q,\alpha,\beta,p,q'}(1-k)^{\frac{q}{p}}\\
&\leq C^q_{Q,\alpha,\beta,p,q'}.
\end{split}\end{equation}
Hence the above inequality must be equality. Since $q>p$,  if
$\|g(v)|v|^{\beta}\|_p^p=k<1,$ we must have $\mu=\delta_{u_0}$, $\nu=C^q_{Q,\alpha,\beta,p,q'}\delta_{u_0}$ and $g=0$.
Next, we claim this is impossible to happen. We discuss this by distinguishing two case.
\medskip

\emph{Case 1.}
If $u_0=0$, then $\int_{B_{\mathbb{H}}(0,1)}d\mu=1$, which arrives at a contradiction with the initial hypotheses
$$\int_{B_\mathbb{H}(0,1)}|g_i(v)|^p|v|^{\beta p}dv=\frac{1}{2}.$$

\emph{Case 2.} if $u_0\neq 0$, we will show that
 there exists some $\delta>0$ such that
$$0 \notin B_\mathbb{H}(u_0,\delta)\ \  {\rm and}\ \
\lim_{i\rightarrow \infty}\int_{B_\mathbb{\mathbb{H}}(u_0,\delta)}|I_{\lambda}(g_i)|^q|u|^{-\alpha q}du=0,$$
which arrive at a contradiction with
$\int_{B_{\mathbb{H}}(u_0,\delta)}d\nu=C^q_{Q,\alpha,\beta,p,q'}$.
In fact, since $$I_\lambda (g_i)(u)\rightarrow I_\lambda(g)(u)=0\ \ \forall\ u\in \mathbb{H}^n,$$
then it follows from \eqref{HSW} that
\begin{equation}\label{x1}\begin{split}
\int_{B_\mathbb{H}(u_0,\delta)}|I_{\lambda}(g_i)|^t|u|^{-\alpha t}du &\lesssim \int_{B_\mathbb{H}(u_0,\delta)}|I_{\lambda}(g_i)|^tdu\\
&\lesssim \int_{\mathbb{H}^n}|I_{\lambda}(g_i)|^tdu\\
&\lesssim \big(\int_{\mathbb{H}^n}|g_i|^p|v|^{\beta p}dv\big)^{\frac{t}{p}},
\end{split}\end{equation}
where $\frac{1}{t}=\frac{1}{p}+\frac{\lambda+\beta}{Q}-1$.
This together with the fact $\alpha>0$, $\frac{1}{q}=\frac{1}{p}+\frac{\lambda+\alpha+\beta}{Q}-1$ leads to $t>q$. Combining
the inequality \eqref{x1} and $I_\lambda (g_i)(u)\rightarrow 0$ a.e. $u\in \mathbb{H}^n$, we derive that
$$\lim_{i\rightarrow \infty}\int_{B_\mathbb{H}(u_0,\delta)}|I_{\lambda}(g_i)|^q|u|^{-\alpha q}du=0.$$
This accomplishes the proof of Theorem \ref{thm1}.

\section{The Proof of Theorem \ref{thm3}}
In this section, we will give the existence results about the Stein-Weiss inequalities with $|z|$ weights on the Heisenberg group. Namely, we will give the proof of
Theorem \ref{thm3}. Since the idea of proof is similar to that of Theorem \ref{thm1} in principle, we will give an approximately outline without detailed explanation.
\medskip

It is not hard to check that extremals of the inequality \eqref{HSW1} are those solving maximizing problem
\begin{equation}\label{mini1}
C_{Q,\alpha,\beta,p,q'}:= \sup \{\|I_{\lambda}(g)|z|^{-\alpha}\|_{L^q(\mathbb{H}^n)} : g\geq 0 , \|g|z'|^{\beta}\|_{L^p(\mathbb{H}^{n})}=1\}.
\end{equation}
Since $\sup \{\|I_{\lambda}(g)|z|^{-\alpha}\|_{L^q(\mathbb{H}^n)} : g\geq 0 , \|g|z'|^{\beta}\|_{L^p(\mathbb{H}^{n})}=1\}$ is dilation-invariant
and translation-invariant about $v_0=(0,t_0)$, hence we may suppose that $\{g_i\}_i$
is also a maximizing sequence for problem \eqref{mini1}
with
\begin{equation}\label{dt}
\int_{B_\mathbb{H}(0,1)}|g_i|^p|z'|^{\beta p}dv=\sup_{t\in \mathbb{R}}\int_{B_\mathbb{H}((0,t),1)}|g_i|^p|z'|^{\beta p}dv=\frac{1}{2}
\end{equation}
for all $i\in \mathbb{N}$. We define the measure:
$$\mu_i:=|g_i|^p|z'|^{\beta p}dv \ \ and\ \ \nu_i:=|I_{\lambda}(g_i)|^q|z|^{-\alpha q}du.$$
Then there exist two bounded measures $\mu$ and $\nu$ such that
$$\mu_i\rightharpoonup \mu \ \ and\ \ \ \nu_i\rightharpoonup \nu \ \ weakly \ in\  the\ sense \ of\  measures\  as\  i\rightarrow\infty.$$
Easily, it follows from lower semi-continuity of the measure that
$$\int_{\mathbb{H}^n}d\mu\leq \lim_{i\rightarrow +\infty}\int_{\mathbb{H}^n}d\mu_i=1\ \ {\rm and}\ \ \int_{\mathbb{H}^n}d\nu\leq \lim_{i\rightarrow +\infty}\int_{\mathbb{H}^n}d\nu_i=C^q_{Q,\alpha,\beta,p,q'}.$$
We now apply Lions' first concentration compactness to the sequence of measure $\{\mu_i\}$.
\medskip
\begin{lemma}
There exists a subsequence of  $\{\mu_i\}$ such that one of the following condition holds:
(a)(Compactness) There exists $\{v_i\}$ in $\mathbb{H}^n$ such that for each $\epsilon>0$ small enough, we can find $R_{\epsilon}>0$
with $$\int_{B_\mathbb{H}(v_i,R_\varepsilon)}d\mu_i\geq1-\varepsilon \ \ {\rm for}\ {\rm all}\ i.$$
(b)(Vanishing) For all $R>0$, there holds:
$$\lim_{i\rightarrow\infty}(\sup_{v\in \mathbb{H}^n}\int_{B_{\mathbb{H}}(v,R)}d\mu_i)=0 .$$
(c)(Dichotomy) There exist $k \in(0,1)$ such that for any $\varepsilon >0$, there exists $R_\varepsilon>0$ and a sequence $\{v_i^\varepsilon\}_{i\in \mathbb{N}}$, with the following property: given $R'>R_\varepsilon$, there are nonnegative measures $\mu_i^1$ and $\mu_i^2$ such that
$$0\leq \mu_i^1+\mu_i^2\leq \mu_i,\ \ Supp(\mu_i^1)\subset  B_{\mathbb{H}}(v_i^\varepsilon,R_\varepsilon),
\ \  Supp(\mu_i^2)\subset \mathbb{H}^n\setminus B_{\mathbb{H}}(v_i^\varepsilon,R'),$$
$$ \mu_i^1=\mu_i|_{B_{\mathbb{H}}(v_i^\varepsilon,R_\varepsilon)},
\ \ \mu_i^2= \mu_i|_{\mathbb{H}^n\setminus B_{\mathbb{H}}(v_i^\varepsilon,R')},$$
$$\limsup_{i\rightarrow\infty}\Big(|k- \int_{\mathbb{H}^n}d\mu_i^1|+|(1-k)-\int_{\mathbb{H}^n}d\mu_i^2|\Big)\leq \varepsilon.$$
\end{lemma}

\begin{lemma}\label{lem31}
Suppose that $\{g_i\}_i$ is a maximizing sequence for problem \eqref{mini1}, then compactness (a) holds. In particular, we also have that $\int_{\mathbb{H}^n}d\mu=1$.
\end{lemma}
\begin{proof}
Clearly, (b) can not occur because $\int_{B_\mathbb{H}(0,1)}|g_i|^p|z'|^{\beta p}dv=\frac{1}{2}$. Next, we only need to show that (c) also does not happen.
We argue this by contradiction. Suppose that (c) occurs, then for any $\varepsilon>0$, we can find $R_{0}>0$ and $\{v_i\}\subset \mathbb{H}^n$ such that for any given $R\geq R_0$, there holds
\begin{equation}\label{d1}
\|g_i|z'|^{\beta}\chi_{B_\mathbb{H}(v_i,R)}\|_p^p=k+ O(\varepsilon)\ \ {\rm and}\ \ \|g_i|z'|^{\beta}\chi_{B_\mathbb{H}^c(v_i,R)}\|_p^p=1-k+ O(\varepsilon).
\end{equation}
Direct computations yield that
\begin{equation}\begin{split}
&|I_{\lambda}(g_i)(u)-I_{\lambda}(g_i\chi_{B_\mathbb{H}(v_i,R)})(u)|\\
&\ \ =|\int_{\mathbb{H}^n}\frac{(g_i\chi_{B_\mathbb{H}^c(v_i,R)})(v)}{|u^{-1}v|^{\lambda}}dv|\\
&\ \ \leq \big(\int_{B_\mathbb{H}^c(v_i,R)}|g_i(v)|z'|^{\beta}|^pdv\big)^{\frac{1}{p}}\big(\int_{B_\mathbb{H}^c(v_i,R)}|u^{-1}v|^{-\lambda p'}|z|^{-\beta p'}dv\big)^{\frac{1}{p'}},\\
\end{split}\end{equation}
here $v=(z,t)$. under the assumptions of Theorem \ref{thm3}, we obtain $(\lambda+\beta)p'>Q$ and $\beta p'<2n$. On the other hand��
$$ \big(\int_{B_\mathbb{H}^c(v_i,R)}|u^{-1}v|^{-\lambda p'}|z|^{-\beta p'}dv\big)^{\frac{1}{p'}}\lesssim
\int_{R}^{\infty}r^{Q-1-(\lambda+\beta)p'}dr\int_{\Sigma}|z^{*}|^{-\beta p'}d\mu(\xi^{*}),$$
where $\Sigma=\{u\in \mathbb{H}^n:\ |u|=1\}$ and $\xi^{*}=(z^*, t^*)\in \Sigma$.
Hence, $$\lim_{R\rightarrow \infty}I_{\lambda}(g_i\chi_{B_\mathbb{H}(v_i,R)})(u)=I_{\lambda}(g_i)(u)\ \ {\rm for}\ u\in \mathbb{H}^n.$$

Now we can apply the Brezis-Lieb lemma to obtain
$$\|I_{\lambda}(g_i)|z|^{-\alpha}\|_{q}^q=\|I_{\lambda}(g_i\chi_{B_\mathbb{H}(v_i,R)})|z|^{-\alpha}\|_{q}^q+
\|I_{\lambda}(g_i\chi_{B_\mathbb{H}^c(v_i,R)})|z|^{-\alpha}\|_{q}^q+o(1).$$
Combining the $|z|$ weighted Stein-Weiss inequality \eqref{HSW1} and condition \eqref{d1}, we conclude that
\begin{equation}\begin{split}
&\|I_{\lambda}(g_i\chi_{B_\mathbb{H}(v_i,R)})|z|^{-\alpha}\|_{q}^q+
\|I_{\lambda}(g_i\chi_{B_\mathbb{H}^c(v_i,R)})|z|^{-\alpha}\|_{q}^q+o(1)\\
&\ \ \leq C_{Q,\alpha,\beta, p,q'}^q\|g_i\chi_{B_\mathbb{H}(v_i,R)}|z'|^{\beta}\|_p^q+C_{Q,\alpha,\beta, p,q'}^q\|g_i\chi_{B_\mathbb{H}^c(v_i,R)}|z'|^{\beta}\|_p^q+o(1)\\
&\ \ \leq C_{Q,\alpha,\beta, p,q'}^q(k^{\frac{q}{p}}+(1-k)^{\frac{q}{p}})+O(\varepsilon)+o(1)\\
&\ \ < C_{Q,\alpha,\beta, p,q'}^q,
\end{split}\end{equation}
which arrives at a contradiction with $\lim\limits_{i\rightarrow \infty}\|I_{\lambda}(g_i)|z|^{-\alpha}\|_{q}^q=C_{Q,\alpha,\beta, p,q'}^q$.
\medskip

Now, we claim that $\int_{\mathbb{H}^n}d\mu=1$.
It follows from $\int_{B_\mathbb{H}(0,1)}|g_i|^p|z'|^{\beta p}dv=\frac{1}{2}$ that
$\{v_i\}$ must be a bounded sequence in $\mathbb{H}^n$.  Hence, we may assume for any $\varepsilon>0$, there exist $R_0$ such that for any $R\geq R_0$, there holds $$\int_{B_\mathbb{H}(0,R)}d\mu_i\geq1-\varepsilon.$$
Then choose $\phi \in C^{\infty}_c(\mathbb{H}^n)$ with $\phi|_{B_\mathbb{H}(0,R)}=1$, then by the weak convergence in the sense of measure, we derive that
\begin{equation}\begin{split}
\int_{\mathbb{H}^n}d\mu&\geq \int_{\mathbb{H}^n}\phi(u)d\mu\\
&\geq \lim_{i\rightarrow \infty}\int_{B_\mathbb{H}(0,R)}|g_i|^p|z'|^{\beta p}dv\\
&\geq \lim_{i\rightarrow \infty}\int_{B_\mathbb{H}(0,R_0)}d\mu_i
\geq 1-\varepsilon.\\
\end{split}\end{equation}
Since $\varepsilon$ is arbitrarily small, we derive that $\int_{\mathbb{H}^n}d\mu \geq 1$. This accomplishes the proof of Lemma \ref{lem31}.
\end{proof}

The following lemma states that
$$I_{\lambda}(g_{i})(u)|z|^{-\alpha}\rightarrow I_{\lambda}(g)(u)|z|^{-\alpha}\ \  a.e.$$
Since the proof is similar to that of Lemma \ref{lem4}, we omit the detailed proof.
\begin{lemma}\label{lem41}
Let $\{g_{i}\}_i$ be a maximizing sequence satisfying
$$\|g_{i}(v)|z'|^{\beta}\|_{L^{p}(\mathbb{H}^n)}=1\ \ {\rm and}\ \ \int_{B_{\mathbb{H}}(0,R)}|g_{i}(v)|^{p}|z'|^{\beta p}dv\geq1-\varepsilon(R).$$
we may assume that $g_{i}(v)|z'|^{\beta}\rightarrow g(v)|z'|^{\beta}$ weakly in $L^{p}(\mathbb{H}^n)$ (by passing to a subsequence if necessary). Then, by passing to a subsequence again  if necessary,
$$I_{\lambda}(g_{i})(u)|z|^{-\alpha}\rightarrow I_{\lambda}(g)(u)|z|^{-\alpha}\ \  a.e..$$
\end{lemma}

\begin{lemma}\label{lem5}
There exist $I\subset \mathbb{N}$ at most countable and a family $\{u_i\}_{i\in I}$ in $\mathbb{H}^n$ such that
$$\nu=|I_{\lambda}(g)(u)|^q|z|^{-\alpha q}du+\sum_{i\in I}\nu^i \delta_{u_i}$$
and $$\mu \geq |g|^p|z'|^{\beta p}dv+\sum_{i\in I}\mu^i \delta_{u_i}$$
where $\nu^i=\nu(u_i)$, $\mu^i=\mu(u_i)$ for all $i\in I$.
Furthermore, we also have $\nu^i\leq C_{Q,\alpha,\beta,p,q'}^q(\mu^i)^{\frac{q}{p}}$.
\end{lemma}

\emph{The Proof of Theorem \ref{thm3}:}
Without loss of generality, we may assume that
$\alpha>0$. Once we obtain $\|g|z'|^{\beta}\|_p=1$, according to \eqref{HSW1} and $g_i|z'|^{\beta}\rightharpoonup g|z'|^{\beta}$ weakly in $L^p(\mathbb{H}^n)$, we derive that $g_i|z'|^{\beta}\rightarrow g|z'|^{\beta}$ strong in $L^p(\mathbb{H}^n)$ and $\|I_{\lambda}(g)|z|^{-\alpha}\|_q=C_{Q,\alpha,\beta,p,q'}$, which implies that  $g$ is a extremal function for the $|z|$ weighted Stein-Weiss inequality on Heisenberg group. Hence, we only need to prove that $\|g|z'|^{\beta}\|_p=1$. We argue this by contradiction. If
$$\|g|z'|^{\beta}\|_p^p=k<1,$$ then it follows from $\mu(\mathbb{H}^n)=1$ and $\nu(\mathbb{H}^n)=C^q_{Q,\alpha,\beta,p,q'}$ and Lemma \ref{lem5} that
\begin{equation}\begin{split}
\nu(\mathbb{H}^n)&=\|I_{\lambda}(g)|z|^{-\alpha}\|_q^q+\sum_{i\in I}\nu^i\\
&\leq C^q_{Q,\alpha,\beta,p,q'}\|g|z'|^{\beta}\|_p^q+\sum_{i\in I} C^q_{Q,\alpha,\beta,p,q'}(\mu^i)^{\frac{q}{p}}\\
&\leq C^q_{Q,\alpha,\beta,p,q'}k^{\frac{q}{p}}+ C^q_{Q,\alpha,\beta,p,q'}(\sum_{i\in I}\mu^i)^{\frac{q}{p}}\\
&\leq C^q_{Q,\alpha,\beta,p,q'}.
\end{split}\end{equation}
Hence the above inequality must be equality. Since $q>p$, we must have $\mu=\delta_{u_0}$, $\nu=C^q_{Q,\alpha,\beta,p,q'}\delta_{u_0}$ and $g=0$.
Next, we claim this is impossible to happen. We discuss this by distinguishing two case.
\medskip

\emph{Case 1.}
If $u_0=(z_0, t_0)$ with $z_0=0$, then $\int_{B_{\mathbb{H}}((0,t_0),1)}d\mu=1$, which arrives at a contradiction with the initial hypotheses
$$\sup_{t\in \mathbb{R}}\int_{B_\mathbb{H}((0,t),1)}|g_i|^p|z'|^{\beta p}du=\frac{1}{2}.$$

\emph{Case 2.} if $u_0=(z_0, t_0)$ with $z_0\neq0$, we will show that
 there exists some $\delta>0$ such that
$$(0,t)\notin B_\mathbb{H}(u_0,\delta)\ \  {\rm and}\ \
\lim_{i\rightarrow \infty}\int_{B_\mathbb{\mathbb{H}}(u_0,\delta)}|I_{\lambda}(g_i)|^q|z|^{-\alpha q}du=0,$$
which arrive at a contradiction with
$\int_{B_{\mathbb{H}}(u_0,\delta)}d\nu=C^q_{Q,\alpha,\beta,p,q'}$.
In fact, since $$I_\lambda (g_i)(u)\rightarrow I_\lambda(g)(u)=0,\ \ \forall\ u\in \mathbb{H}^n,$$
In view of the $|z|$ weighted Stein-Weiss inequality on Heisenberg group, one can derive that
\begin{equation}\begin{split}
\int_{B_\mathbb{H}(u_0,\delta)}|I_{\lambda}(g_i)|^t|z|^{-\alpha t}du &\lesssim \int_{B_\mathbb{H}(u_0,\delta)}|I_{\lambda}(g_i)|^tdu\\
&\lesssim \int_{\mathbb{H}^n}|I_{\lambda}(g_i)|^tdu\\
&\lesssim \big(\int_{\mathbb{H}^n}|g_i|^p|z'|^{\beta p}dv\big)^{\frac{t}{p}},
\end{split}\end{equation}
where $\frac{1}{t}=\frac{1}{p}+\frac{\lambda+\beta}{Q}-1$.
This together with the fact $\alpha>0$, $\frac{1}{q}=\frac{1}{p}+\frac{\lambda+\alpha+\beta}{Q}-1$ and $I_\lambda (g_i)(u)\rightarrow 0$ a.e. $u\in \mathbb{H}^n$ yields
$$\lim_{i\rightarrow \infty}\int_{B_\mathbb{H}(u_0,\delta)}|I_{\lambda}(g_i)|^q|z|^{-\alpha q}du=0.$$
Then we accomplish the proof of Theorem \ref{thm3}.

\section{The Proof of Theorem \ref{thm5}}
In this section, we will use the Sawyer and Wheeden condition on weighted inequalities for integral operators in \cite{SW} and concentration compactness principle to establish inequality \eqref{HSWF} and existence of their extremal functions.
\medskip
Observe that inequality \eqref{HSWF} is equivalent to the boundedness of the following weighted fractional integral operator:
$$\|T(g)|x'|^{-\alpha}\|_{L^q(\mathbb{R}^{m+n})}\leq C\|g|y'|^{\beta}\|_{L^p(\mathbb{R}^{m+n})},$$
where $T(g)$ is defined as $$T(g)=\int_{\mathbb{R}^{n+m}}|x-y|^{-\lambda}g(y)dy.$$ In order to prove that operator $T$ is bounded from $L^{p}(|y'|^{\beta p}dy, \mathbb{R}^{m+n})$ to $L^{q}(|x'|^{-\alpha q}dx, \mathbb{R}^{m+n})$, according to Sawyer and Wheeden's result \cite{SW}, we only need to verify that the following two condition hold.
\begin{lemma}\label{Saw}
The operator $T$ is bounded from $L^{p}(|y'|^{\beta p}dy, \mathbb{R}^{m+n})$ to $L^{q}(|x'|^{-\alpha q}dx, \mathbb{R}^{m+n})$ if and only if the following two conditions hold.

(1) There exists $\varepsilon>0$ such that for any pair of balls $B$ and $B'$ with radius $r$ and $r'$ satisfying $B'\subset 4B$,
$$\big(\frac{r'}{r}\big)^{Q-\varepsilon}\big(\frac{\phi(B')}{\phi(B)}\big)\leq C_\varepsilon.$$

(2) There exists $t>1$ such that for any ball $B\subset \mathbb{R}^{n+m}$,
$$\phi(B)|B|^{\frac{1}{q}+\frac{1}{p'}}\big(\frac{1}{|B|}\int_{B}|x'|^{-\alpha q t}dx\big)^{\frac{1}{qt}}\big(\frac{1}{|B|}\int_{B}|y'|^{-\beta p't }dy\big)^{\frac{1}{p't}}\leq C_t.$$
\end{lemma}
Here $\phi(B)=2^{24\lambda}r^{-\lambda}$ and $B, B' \in \mathbb{R}^{m+n}$.
According to the assumptions of Theorem \ref{thm5}, we see that there exists $t>1$ such that $\alpha qt<m$, $\beta p't<m$.
Then it follows that integrals $\int_{B}|x'|^{-\alpha q t}dx$ and $\int_{B}|x'|^{-\beta p't }dx$ are finite. Since the ideas of existence of extremals for the inequality \eqref{HSWF} is similar to the proof of Theorem \ref{thm3} in principle, we omit the detailed proof.
Then we accomplish the proof of Theorem \ref{thm5}.

\section{The Proof of Theorem \ref{thm6}}
In this section, we will use the method of moving plane in integral forms introduced by Chen, Li and Ou \cite{CLO} to prove that
symmetry results for each pair of solutions $(u,v)$ of integral system \eqref{system3}.
In order to prove our theorem, we first introduce some notations. For $\lambda \in \mathbb{R}$, denote
$$H_{\nu}=\{x\in \mathbb{R}^{m+n}:x_{1}>\nu\}.$$
For $x=(x_{1},\hat{x}) \in \mathbb{R}^{m+n}$, we denote $x_{\lambda}=(2\nu-x_{1},\hat{x})$. Write $u_{\lambda}(x)=u(x^{\nu})$ and $v_{\nu}(x)=v(x^{\nu})$. Assume that $(u,v)\in L^{p_1+1}(\mathbb{R}^{m+n})\times L^{p_2+1}(\mathbb{R}^{n+m})$ is a pair of positive solutions of the following integral system
\begin{equation}\label{weighted}\begin{cases}
u(x)=\int_{\mathbb{R}^{n+m}}|x'|^{-\alpha}|x-y|^{-\lambda} v^{p_2}(y)|y'|^{-\beta} dy,\\
v(x)=\int_{\mathbb{R}^{n+m}}|x'|^{-\beta}|x-y|^{-\lambda} u^{p_1}(y)|y|^{-\alpha} dy,
\end{cases}\end{equation}
where $\frac{1}{p_1-1}+\frac{1}{p_2-1}=\frac{\alpha+\beta+\lambda}{n+m}.$  We only prove that $u(x)|_{\mathbb{R}^m\times \{0\}}$ and $v(x)|_{\mathbb{R}^m\times \{0\}}$ are radially symmetric and monotonously decreasing about origin in $\mathbb{R}^m$. We need the following lemma whose proof can be seen in \cite{CL}.
\begin{lemma}\label{lemma}
If $(u,v)$ is a pair of nonnegative solutions of \eqref{weighted}, then for any $x\in \mathbb{R}^{m+n}$, there holds
\begin{equation}\nonumber\\\begin{split}
&u(x)-u(x^\nu)\\
&\ \ =\int_{H_\nu}\Big(\frac{1}{|(x^\nu)'|^{\alpha}|x^\nu-y|^\lambda}-\frac{1}{|x'|^\alpha|x-y|^{\lambda}}\Big)
\Big(\frac{1}{|(y^\nu)'|^\beta}-\frac{1}{|y'|^\beta}\Big)
v^{p_2}(y)dy\\
&\ \ \ \ +\int_{H_\nu}\Big(\frac{1}{|x'|^\alpha}-\frac{1}{|(x^\nu)'|^\alpha}\Big)\Big(\frac{1}{|x-y|^{\lambda}|(y^\nu)'|^\beta}+
\frac{1}{|x^\nu-y|^{\lambda}|(y^\nu)'|^\beta}\Big)v^{p_2}(y)dy\\
&\ \ \ \ +\int_{H_\nu}\Big(\frac{1}{|(x^\nu)'|^\alpha|x^\lambda-y|^{\lambda}|(y^\nu)'|^\beta}-\frac{1}
{|x'|^\alpha|x-y|^{\lambda}|(y^\nu)'|^\beta}\Big)\Big(v_{\nu}^{p_2}(y)-v^{p_2}(y)\Big)\\
\end{split}\end{equation}
and
\begin{equation}\nonumber\\\begin{split}
&v(x)-v(x^\nu)\\
&\ \ =\int_{H_\nu}\Big(\frac{1}{|(x^\nu)'|^\beta|x^\nu-y|^{\lambda}}-\frac{1}{|x|^\alpha|x-y|^{\lambda}}\Big)
\Big(\frac{1}{|(y^\nu)'|^\alpha}-\frac{1}{|y'|^\alpha}\Big)
u^{p_1}(y)dy\\
&\ \ \ \ +\int_{H_\nu}\Big(\frac{1}{|(x^\nu)'|^\alpha}-\frac{1}{|(x^\nu)'|^\beta}\Big)\Big(\frac{1}{|x-y|^{\lambda}|(y^\nu)'|^\alpha}+
\frac{1}{|x^\nu-y|^{\lambda}|(y^\nu)'|^\alpha}\Big)u^{p_1}(y)dy\\
&\ \ \ \ +\int_{H_\nu}\Big(\frac{1}{|(x^\nu)'|^\beta|x^\nu-y|^{\lambda}|(y^\nu)'|^\alpha}-\frac{1}
{|x'|^\beta|x-y|^{\lambda}|(y^\nu)'|^\alpha}\Big)\Big(u^{p_1}_{\nu}(y)-u^{p_1}(y)\Big)dy.
\end{split}\end{equation}
\end{lemma}
Now we are in a position to prove Theorem \ref{thm6}, the proof will be divided into two steps.

\medskip

\emph{Step 1.} We show that for sufficiently negative $\nu$,
\begin{equation}\label{bb}
u_\nu(x)-u(x)\leq0,\,\,v_\nu(x)-v(x)\leq0,\,\,\,\,\forall \, x\in H_\nu.
\end{equation}
Define
$$H_\nu^u=\{\xi\in H_\nu | u_\nu(x)-u(x)>0\},\,\,\,\,H_\nu^v=\{x\in H_\nu| v_\nu(x)-v(x)>0\}.$$
We only need to verify that $H_\nu^u$ and $Q_\nu^v$ must be empty for sufficiently negative $\nu$.

\medskip

Applying the mean value theorem and Lemma \ref{lemma}, we obtain that for any $x \in H_{\nu}^u$, there holds
\begin{equation}\nonumber\\\begin{split}
u_{\nu}(x)-u(x)&\leq\int_{H_\nu} \big(|x'|^{-\alpha}|x-y|^{-\lambda}|(y^\nu)'|^{-\beta}-|(x^\nu)'|^{-\alpha}|x-y^\nu|^{-\lambda}|(y^\nu)'|^{-\beta}\big)\big(v_{\nu}^{p_2}(y)-v^{p_2}(y)\big)dy\\
&\leq \int_{H_\nu^v}|x|^{-\alpha}|x-y|^{-\lambda}|y'|^{-\beta}\big(v_{\nu}^{p_2}(y)-v^{p_2}(y)\big)dy\\
&\leq q_0\int_{H_\nu^v}|x'|^{-\alpha}|x-y|^{-\lambda}|y'|^{-\beta}v_{\nu}^{p_2-1}(v_{\nu}(y)-v(y))dy
\end{split}\end{equation}
and
\begin{equation}\nonumber\\\begin{split}
v_{\lambda}(x)-v(x)&\leq \int_{H_\lambda} \big(|x|^{-\beta}|x-y|^{-\lambda}|(y^\nu)'|^{-\alpha}-|(x^\nu)'|^{-\beta}|x-y^\nu|^{-\lambda}|(y^\nu)'|^{-\alpha}\big)\big(u_{\nu}^{p_1}(y)-u^{p_1}(y)\big)dy\\
&\leq \int_{H_\nu^u}|x'|^{-\beta}|x-y|^{-\lambda}|y'|^{-\alpha}\big(u_{\nu}^{p_1}(y)-u^{p_0}(y)\big)dy\\
&\leq p_0\int_{H_\nu^u}|x'|^{-\beta}|x-y|^{-\lambda}|y'|^{-\alpha}u_{\nu}^{p_1-1}(y)(u_{\nu}(y)-u(y))dy.
\end{split}\end{equation}
Thanks to $(u,v)\in L^{p_0+1}(\mathbb{R}^{m+n})\times L^{p_2+1}(\mathbb{R}^{m+n})$ and integral inequality \eqref{HSWF}, we derive that
\begin{equation}\label{A1}\begin{split}
\|u_\lambda-u\|_{L^{p_1+1}(H^u_\nu)}&\leq p_2 \|v_\nu\|_{L^{p_2+1}(Q_{\lambda})}^{p_2-1}\|v_{\nu}-v\|_{L^{p_2+1}(H_{\nu}^v)}
\end{split}\end{equation}
and
\begin{equation}\label{A2}\begin{split}
\|v_\nu-v\|_{L^{p_2+1}(H^v_\nu)}&\leq p_1 \|u_\nu\|_{L^{p_1+1}(H_{\nu})}^{p_1-1}\|u_{\nu}-u\|_{L^{p_1+1}(H_{\nu}^u)}.
\end{split}\end{equation}
This together with  \eqref{A1} and \eqref{A2} yields that
\begin{equation}\label{A3}\begin{split}
\|u_\nu-u\|_{L^{p_1+1}(H^u_\nu)}\leq p_1p_2 \|v_\lambda\|_{L^{p_2+1}(H_{\nu})}^{p_2-1}\|u_\nu\|_{L^{p_1+1}(H_{\nu})}^{p_1-1}\|u_{\nu}-u\|_{L^{p_1+1}(H_{\nu}^u)}
\end{split}\end{equation}
and
\begin{equation}\label{A4}\begin{split}
\|v_\nu-v\|_{L^{p_2+1}(H^v_\nu)}\leq p_1q_1 \|v_\nu\|_{L^{p_2+1}(H_{\nu})}^{p_2-1}\|u_\nu\|_{L^{p_1+1}(H_{\nu})}^{p_1-1}\|v_{\nu}-v\|_{L^{p_2+1}(H_{\nu}^u)}
\end{split}\end{equation}

\medskip

In virtue of the conditions $(u,v)\in L^{p_1+1}(\mathbb{R}^{m+n})\times L^{p_2+1}(\mathbb{R}^{m+n})$, we can choose sufficiently negative $\nu$ such that
\begin{equation}\label{A5}\begin{split}
\|u_\nu-u\|_{L^{p_1+1}(H^u_\nu)}\leq \frac{1}{2}\|u_\nu-u\|_{L^{p_1+1}(H^u_\nu)},\ \ \|v_\nu-v\|_{L^{p_2+1}(H^v_\nu)}\leq
\frac{1}{2} \|v_\nu-v\|_{L^{p_2+1}(H^v_\nu)},
\end{split}\end{equation}
which implies that $H_{\nu}^u$ and $H_{\nu}^v$ must be  must be empty sets.

\medskip

\emph{Step 2.} The inequality \eqref{bb} provides a starting point for moving plane.
Now we start from the negative infinity of $x_1$-axis and move the plane to the right as long as \eqref{bb} holds. Define
$$\nu_0=\sup\{\nu\  |\  u_\mu(x)\leq u(x),\,\,v_\mu(x)\leq v(x),\,\,\mu\leq\nu,\,\,\forall\  x\in H_\mu\ \}.$$
We will show that
$\\nu_0=0$.
Suppose on the contrary that $\nu_0<0$, we will show that $u$ and $v$ must be symmetric about the plane $x_1=\nu_0$, namely
\begin{equation}\label{ba}
u_{\nu_0}(x)\equiv u(x),\,\,\,v_{\nu_0}(x)\equiv v(x),\,\,\,\forall\  x\in  H_{\nu_{0}}.
\end{equation}
Otherwise, we may assume on $H_{\nu_0}$,
\begin{equation}\nonumber
u_{\nu_0}(x)\leq u(x),\,\, v_{\nu_0}(x)\leq v(x),\,\,\text{but}\,\,u_{\nu_0}(x)\not\equiv u(x)\,\,\text{or}
\,\,v_{\nu_0}(x)\not\equiv v(x).
\end{equation}
In the case of $u_{\nu_0}(\xi)\not\equiv u(x)$ on $H_{\nu_0}$, one can employ Lemma \ref{lemma} to obtain $u_{\nu_0}(x)<u(x)$ and $v_{\lambda_0}(x)<v(x)$ in the interior of $H_{\nu_0}$.

\medskip

Next, we will show that the plane can be moved further to the right. That is to say, there exists an $\varepsilon>0$ such that
for any $\nu\in[\nu_0,\nu_0+\varepsilon)$,
\begin{equation}\label{move}
u_{\nu}(x)\leq u(x) ,\,\,v_{\nu}(x)\leq v(x) ,\,\,\,\forall\  x\in H_{\nu}.
\end{equation}
Set
$$\overline{H_{\nu_0}^u}=\{x\in H_{\nu_0}|u(x)\geq u_{\nu_0}(\xi)\},\,\,\,
\overline{Q_{\nu_0}^v}=\{x\in H_{\nu_0}|v(x)\geq v_{\nu_0}(x)\}.$$
It is clear that both $\overline{H_{\nu_0}^u}$ and $\overline{H_{\nu_0}^v}$ must be measure zero and
$$\lim_{\nu\to \nu_0}H_\nu^u\subset\overline{H_{\nu_0}^u}\ \ \ \lim_{\nu\to \nu_0}
Q_\nu^v\subset\overline{H_{\nu_0}^v}.$$ Combining this and
integrability conditions $(u,v)\in L^{p_1+1}(\mathbb{R}^{m+n})\times L^{p_2+1}(\mathbb{R}^{m+n})$, one can choose $\varepsilon$ small enough such that
for all $\lambda\in[\lambda_0,\lambda_0+\varepsilon)$, there holds
\begin{equation}\nonumber\\\begin{split}
p_1p_2\|u\|_{L^{p_1+1}(H_{\nu}^u)}^{p_1-1}\|v\|_{L^{p_2+1}(H_{\nu}^v)}^{p_2-1}\leq \frac{1}{2}.
\end{split}\end{equation}
As what we did in the estimates of \eqref{A5}, one can obtain
$$\|u-u_\nu\|_{L^{p_1+1}(H_\nu^u)}=\|v-v_\nu\|_{L^{p_2+1}(H_\nu^v)}=0,$$
which implies that $H_\nu^u$ and $H_\nu^v$ must be measure zero.
\medskip

Finally, we show that the plane cannot stop before touching the origin. We argue this by contradiction. Suppose the plane stops at $x_1=\nu_0<0$, according to the above argument, we know that $u$ and $v$ must be symmetric about the plane $x_1=\nu_0$, that is to say
\begin{equation*}
u_{\nu_0}(x)\equiv u(x),\,\,\,v_{\nu_0}(x)\equiv v(x),\,\,\,\forall\  x \in  H_{\nu_{0}}.
\end{equation*}
Since $|x-y^{\nu_0}|^{-\lambda}<|x-y|^{-\lambda}$, $|x^{\nu_0}-y|^{-\lambda}<|x-y|^{-\lambda},$ $|(x^{\nu_0})'|>|x'|$ and $|(y^{\nu_0})'|>|y'|$, we derive that
\begin{equation*}\begin{split}
u_{\nu_0}(x)-u(x)&=\int_{H_{\nu_{0}}}(|(x^{\nu_0})'|^{-\alpha}|x^{\nu_0}-y|^{-\lambda}-|x'|^{-\alpha}|x-y|^{-\lambda})v^{p_2}(y)|y'|^{-\beta}dy\\
&\ \ +\int_{H_{\nu_0}}(|(x^{\nu_0})'|^{-\alpha}|x^{\nu_0}-y^{\nu_0}|^{-\lambda}-|x'|^{-\alpha}|x-y^{\nu_0}|^{-\lambda})v_\lambda^{p_2}(y)|(y^{\nu_0})'|^{-\beta}dx\\
&< \int_{H_{\nu_{0}}}(|x'|^{-\alpha}|x^{\nu_0}-y^{\nu_0}|^{-\lambda}-|x'|^{-\alpha}|x-y|^{-\lambda})v^{p_2}(y)|y'|^{-\beta}dy\\
&\ \ +\int_{H_{\nu_0}}(|x'|^{-\alpha}|x-y|^{-\lambda}-|x'|^{-\alpha}|x^{\nu_0}-y|^{-\lambda})v^{p_2}(x)|(y^{\nu_0})'|^{-\beta}dy\\
&=\int_{H_{\nu_0}}|x'|^{-\alpha}\big(|x-y|^{-\lambda}-|x^{\nu_0}-y|^{-\lambda})v^{p_2}(y)(|(y^{\nu_0})'|^{-\beta}-|y|^{-\beta})dy\\
&<0,
\end{split}\end{equation*}
which leads to a contradiction. Then it follows that $\nu_0=0$, $u_{0}(x)\leq u(x)$ and  $v_{0}(x)\leq v(x)$.
Carrying out the above procedure in the opposite direction, one can also obtain the moving plane must stop before the origin. Hence,
we have $u_{0}(x)\geq u(x)$ and  $v_{0}(x)\geq v(x)$. Since the $x_1$ direction can be replace by
the $x_i$ direction for $i=1,2,\cdot\cdot\cdot m$, we deduce that $u(x)|_{\mathbb{R}^m\times \{0\}}$ and $v(x)|_{\mathbb{R}^m\times \{0\}}$ must be radially symmetric and monotone decreasing about the origin in $\mathbb{R}^m$. This accomplishes the proof of Theorem \ref{thm6}.

{\bf Acknowledgement:} The authors wish to thank the referees for their helpful comments which improve the exposition of the paper.

\end{document}